\documentclass[12pt]{amsart}
\usepackage[margin=1.0in]{geometry}
\usepackage{ifxetex,ifluatex}
\usepackage{makecell}

\newif\ifxetexorluatex
\ifxetex
  \xetexorluatextrue
\else
  \ifluatex
    \xetexorluatextrue
  \else
    \xetexorluatexfalse
  \fi
\fi

\ifxetexorluatex
  \usepackage{fontspec}
\else
  \usepackage[T1]{fontenc}
  \usepackage[utf8]{inputenc}
  \usepackage{lmodern}
\fi
\usepackage{hyperref, amsmath, amssymb, mathtools}

\usepackage{amsthm}
\usepackage{amsfonts}
\usepackage{amsmath}
\usepackage{comment}
\usepackage{chngcntr}
\usepackage{ctable}
\newtheorem{thm}{Theorem}
\newtheorem*{thm*}{Theorem}
\newtheorem{lemma}[thm]{Lemma}

\theoremstyle{definition}

\newtheorem*{remark}{Remark}
\numberwithin{thm}{section}
\newcommand{\p}{\mathfrak{p}}
\renewcommand{\P}{\mathfrak{P}}
\renewcommand{\O}{\mathcal{O}}
\newcommand{\Q}{\mathbb{Q}}
\newcommand{\Z}{\mathbb{Z}}
\renewcommand{\l}{\ell}
\newcommand{\Ss}{\mathcal{S}}

\newcommand{\set}[1]{\left\{ #1 \right\}}
\newcommand{\artin}[1]{\left[\frac{L/K}{#1}\right]}
\DeclareMathOperator{\Li}{Li}
\DeclareMathOperator{\Gal}{Gal}

\title{Formulas for Chebotarev densities of Galois extensions of number fields}
\author[N. Sweeting]{Naomi Sweeting}
\address{Department of Mathematics, University of Chicago, Chicago, IL 60637}
\email{nsweeting@uchicago.edu}
\author[K. Woo]{Katharine Woo}
\address{Department of Mathematics, Stanford University, Stanford, CA 94305}
\email{katywoo@stanford.edu}

\date{June 2018}

\begin{document}

\maketitle

\begin{abstract}
We generalize the Chebotarev density formulas of Dawsey (2017) and Alladi (1977) to the setting of arbitrary finite Galois extensions of number fields $L/K$.  In particular, if $C \subset G = \Gal(L/K)$ is a conjugacy class, then we establish that the Chebotarev density is the following limit of partial sums of ideals of $K$:
\[
-\lim_{X\rightarrow\infty} \sum_{\substack{2\leq N(I)\leq X \\ I \in S(L/K; C)}} \frac{\mu_K(I)}{N(I)}  = \frac{|C|}{|G|},
\] where $\mu_K(I)$ denotes the generalized M\"obius function and $S(L/K;C)$ is the set of ideals $I\subset \O_K$ such that $I$ has a unique prime divisor $\p$ of minimal norm and the Artin symbol $\artin{\p}$ is $C$. To obtain this formula, we generalize several results from classical analytic number theory, as well as Alladi's concept of duality for minimal and maximal prime divisors, to the setting of ideals in number fields.

\end{abstract}

\section{Introduction}

\noindent Partial sums involving the M\"obius function have historically played a crucial role in analytic number theory, most famously in the proof of the Prime Number Theorem. Such sums are ubiquitous  due to their connection to the Riemann zeta function $\zeta(s)$; for example, we have:
\[ \lim_{s\rightarrow 1^+} \zeta(s)^{-1} = \lim_{X\rightarrow\infty} \sum_{n=1}^X \frac{\mu(n)}{n} = 0. \]

Work of both Alladi \cite{Alladi} and Dawsey \cite{Dawsey} found that certain restricted partial sums of $\frac{\mu(n)}{n}$ have strong ties to Chebotarev densities of primes. To be precise, Alladi showed \cite{Alladi} that if  $\gcd(\ell,k) = 1$, then we have that 
\[
-\lim_{X\rightarrow\infty} \hspace{-0.4cm}\sum_{\substack{2\leq n\leq X\\ p_{\min}(n)\equiv \ell \pmod{k}}}\hspace{-0.7cm} \frac{\mu(n)}{n} = \frac{1}{\varphi(k)},
\]
where $\varphi$ is the Euler totient function and $p_{\min}(n)$ is the smallest prime divisor of $n$. This result can be thought of as a glimpse of Dirichlet's theorem on primes in arithmetic progressions.

Dawsey \cite{Dawsey} recognized this result as the cyclotomic case of a more general result dealing with Galois extensions of $\Q$. To introduce her work, we first set up some notation. If $K/\Q$ is a Galois extension and $p$ is an unramified prime, then we define 
\[
\left[\frac{K/\Q}{p}\right] \coloneqq \left\{\left[\frac{K/\Q}{\p}\right] \colon \p\subset \O_K \textrm{ is a prime ideal above }p\right\},
\]
where $\left[\frac{K/\Q}{\p}\right]$ is the Artin symbol for the Frobenius map, which satisfies 
\[
\alpha \mapsto \alpha^p \mod \p \textrm{ for all }\alpha \in K.
\]
It is well-known  that $\left[\frac{K/\Q}{p}\right]$ is a conjugacy class in $\Gal(K/\Q)$. Dawsey \cite{Dawsey} proved that if $C$ is a conjugacy class of $G = \Gal(K/\Q)$, then:
\[
-\lim_{X\rightarrow \infty} \hspace{-0.2cm}\sum_{\substack{2\leq n\leq X \\ \left[\frac{K/\Q}{p_{\min}(n)}\right] = C}} \hspace{-0.2cm} \frac{\mu(n)}{n} = \frac{|C|}{|G|}.
\]

\begin{remark}
Note that Alladi's result follows from Dawsey's theorem by taking $K=\Q(\zeta_k)$ where $\zeta_k$ is a primitive $k$th root of unity. Then the condition that $p_{\min}(n) \equiv \ell \pmod{k}$ is equivalent to the condition that $\left[\frac{K/\Q}{p_{\min}}\right]$  is the conjugacy class $\l$ in $\Gal(\Q(\zeta_k)/\Q) \simeq (\Z/k\Z)^\times $. 
\end{remark}

We aim to generalize Alladi's and Dawsey's results to arbitrary finite Galois extensions of number fields. Throughout this paper, $L/K$ is a finite Galois extension of number fields, and $G = \Gal(L/K)$. Let $\p\subset \O_K$ be an unramified prime. Then we define
\[
\artin{\p} \coloneqq \left\{\left[\frac{L/K}{\P} \right]\colon \P \subset \O_L \textrm{ is a prime ideal above }\p\right\},
\]
where the Artin symbol $\left[\frac{L/K}{\P} \right] \in G$ denotes the Frobenius element  satisfying:
\[
\alpha \mapsto \alpha^{N(\p)} \mod \P \textrm{ for all }\alpha \in L.
\]
Once again, $\artin{\p}$ is a conjugacy class of $\Gal(L/K)$. We will refer to this as the Artin symbol of $\p$, with no risk of ambiguity. We also define the residue of the Dedekind zeta function $\zeta_K(s)$ at $s=1$ by \begin{equation}\numberwithin{equation}{section}\label{c_K}c_K \coloneqq \lim_{s\to 1^+} (s - 1) \zeta_K(s) = \frac{2^{r_1}\cdot (2\pi)^{r_2}\cdot \textrm{Reg}_K \cdot h_K}{w_K \sqrt{|D_K|}},\end{equation}
where $r_1$ is the number of real embeddings of $K$, $2r_2$ is the number of complex embeddings, $h_K$ is the class number, $\textrm{Reg}_K$ is the regulator, $w_K$ is the number of roots of unity in $K$, and $D_K$ is the discriminant.


Our first theorem generalizes an intermediate result of Dawsey \cite{Dawsey}. Whenever we write $I$  (resp. $\p$), we refer to ideals (resp. prime ideals) of $\O_K$. Also, throughout this paper, $k$ denotes an unspecified positive constant depending on $K$ and $L$. For an ideal $I$, we define the maximal norm by \[M(I) \coloneqq \max_{I \subset \p} N(\p).\]

We are interested in the number of ideals of maximal norm with a certain Artin symbol, defined by \[Q_C(I) \coloneqq \#\set{I \subset \p \, \colon N(\p) =M(I)\textrm{ and } \left[\frac{L/K}{\p}\right] = C}.\] The partial sums of $Q_C(I)$ are the subject of our first theorem.
\begin{thm}\label{thm1}
Let $L/K$ be a finite Galois extension of number fields and let $G = \Gal(L/K)$ be the Galois group. If $C$ is a conjugacy class of $G$, then as $X\to\infty$, we have that
\[
\sum_{2\leq N(I)\leq X} Q_C(I) = c_K\cdot \frac{|C|}{|G|} X + O\left(X \exp\set{-k (\log X)^{1/3}}\right). 
\]
\end{thm}

Alladi \cite{Alladi} developed a duality principle to convert a similar intermediate result into his main theorem. His duality principle uses properties of the M\"{o}bius function to relate functions on maximal and minimal prime divisors. In order to apply similar methods, we first define the M\"obius function for ideals of a number field $K$ as 
\begin{equation}
\mu_K(I) \coloneqq \begin{cases} 1, & \textrm{if } I = \O_K \\ 0, &\textrm{if } I \subset \p^2 \textrm{ for some } \p \\ (-1)^k, &\textrm{if } I = \p_1 \cdots \p_k.\end{cases}
\end{equation} For the remainder of this paper, we will write $\mu(I)$ for $\mu_K(I).$ Many of the well-known identities for the M\"obius function on natural numbers extend to ideals as well, which allows us to prove a version of the duality principle for prime ideals in \S\ref{duality}. 

Armed with our duality result, we obtain an analogue of Alladi's and Dawsey's theorems. First we establish some notation. We call an ideal $I \subset \O_K$ \emph{salient} if it has a unique prime divisor of smallest norm, and for such $I$ we denote this prime divisor by $\p_{\min} (I)$. 
If $C\subset \Gal(L/K)$ is a conjugacy class, then we define \begin{equation}S(L/K; C) \coloneqq \set{I \subset \O_K \text{ a salient ideal}\, \colon \p_{\min}(I) \textrm{ is unramified and } \artin{\p_{\min}(I)} = C}.\end{equation}

\begin{thm}\label{thm2}
Let $L/K$ be a finite Galois extension of number fields and let $G = \Gal(L/K)$ be the Galois group. If $C$ is a conjugacy class of $G$, then we have that 
\[
-\lim_{X\rightarrow\infty} \sum_{\substack{2\leq N(I)\leq X \\ I \in S(L/K;C)}} \frac{\mu(I)}{N(I)}  = \frac{|C|}{|G|}.
\]
\end{thm}

This paper is organized as follows. In \S 2, we  introduce the notation and background information used in our proofs. In particular, in \S \ref{classical} we  introduce some classical number theory results, including Landau's Prime Ideal Theorem and the Chebotarev Density Theorem. In \S \ref{duality}, we  discuss how Alladi's duality principle  extends to the prime ideal case. In \S \ref{integral} and \S \ref{bounds}, we introduce various bounds and asymptotics needed in our proofs. In \S 3 and \S 4, we  prove Theorems 1.1 and 1.2, respectively. Finally, in \S 5, we demonstrate some examples of Theorem 1.2. 

\section{Preliminaries}
\subsection{Duality principle for ideals}\label{duality}
Alladi \cite{Alladi} introduced a duality principle relating functions on maximal and minimal prime divisors via M\"obius inversion. This principle is the conceptual basis for our work in the number field setting.

\begin{thm*}[Alladi, 1977]
Let $p_{\max}(n)$ denote the largest prime divisor of $n$ and let $p_{\min}(n)$ denote the smallest prime divisor. If $f$ is a function defined on the integers with $f(1) = 0$, then we have that
{\[
\sum_{d\mid n} \mu(d) f(p_{\max}(d)) = -f(p_{\min}(n))
\] and
\[
\sum_{d\mid n} \mu(d) f(p_{\min}(d)) = -f(p_{\max}(n)).
\]}
\end{thm*}

Alladi's proof does not fully extend  to the number field setting because it relies on an ordering of primes. Adapting his methods requires significant restrictions on the functions $f$, but the connection between information about maximal and minimal divisors is essentially preserved; we present the case we require for the proof of Theorem 2.

Throughout the remainder of this paper, we fix a Galois extension $L/K$ with Galois group $G$, and a conjugacy class $C\subset G$.
\begin{lemma}
Let $f(I)$ be the indicator function of $S(L/K;C)$. Then for ideals $I$, we have: 
\[
\sum_{J\supset I} \mu(J)f(J) =- Q_C(I).
\]
\end{lemma}

\begin{proof}
Given an ideal $I$, we prime factor $I = \p_1^{e_1} \cdots \p_r^{e_r}$, such that $N(\p_i) \leq N(\p_{i+1})$. Split this prime factorization (preserving ordering) into $I = I_1 \cdots I_k,$ where all prime factors of $I_j$ have the same norm.  Because $f$ depends only on the primes of minimal norm, we have:
\begin{align*}
\sum_{J \supset I} \mu(J) f(J) &= \sum_{j = 1}^k \sum_{\substack{J \supset I_j \\ J \neq \O_K}}\mu(J) f(J)\hspace{-0.2cm} \sum_{J' \supset I_{j+1} \cdots I_k} \mu(J') 
\end{align*} For the same reason as in the integral case, the Dirichlet convolution of $\mu$ and $1$ is the indicator function of $\O_K$. Therefore we can discard almost all terms from the sum:
\[\sum_{J \supset I}\mu(J) f(J) = \sum_{\substack{J\supset I_{k} \\ J \neq \O_K}} \mu(J) f(J) = -Q_C(I),
\] where we notice that $f(J) = 0$ for all $J$ in this sum except prime ideals with $\artin{\p} = C$. 

\end{proof}

\begin{remark}
An analogous result holds if we replace $\artin{\p_{\min}} = C$ with any other condition on $\p_{\min}(I)$, such as splitting completely in $L$ or lying over a particular prime of $\Q$. 
\end{remark}

\subsection{Classical results}\label{classical}
We now review some well-known estimates that will appear in our proofs. We require quantitative versions of both the Prime Ideal Theorem and the Chebotarev Density Theorem. We define prime-counting functions by \begin{equation}\pi(K;X) \coloneqq  \#\set{\p \subset \O_K\, \colon N(\p) \leq X}\end{equation} and \begin{equation}\pi_C(L/K; X) \coloneqq \#\set{\p \subset \O_K\, \colon N(\p) \leq X \textrm{ and } \artin{\p} = C}.\end{equation} Due to Lagarias and Odlyzko \cite{LO}, we have the following version of the Chebotarev Density Theorem.
\begin{thm*}[Lagarias-Odlyzko, 1975]
Let  $L/K$ be a Galois extension and let $C \subset G$ be a conjugacy class. Then there exist  constants $c_1, c_2 > 0$ so that, for large enough $X$, we have:  \[\bigg|\pi_C(L/K; X) - \frac{|C|}{|G|}\Li(X)\bigg| \leq c_1X\exp\set{-c_2\sqrt{\log X}}.\]

\end{thm*} Landau's form of the Prime Ideal Theorem \cite{Landauart} immediately follows from this statement of the Chebotarev Density Theorem. 
\begin{thm*}[Landau, 1903]
If $K$ is a number field, then there exist constants $c_1, c_2 > 0$ such that, for large enough $X$, we have \[\big| \pi(K; X) - \Li (X) \big| \leq c_1X \exp\set{-c_2\sqrt{\log X}}. \]
\end{thm*} We also will need Murty and Van Order's explicit version \cite{Murty} of a classical asymptotic for counting ideals of $K$. 
As $X\to\infty$, we have
\begin{equation}\label{idealcount}\sum_{N(I) \leq X}\hspace{-0.1cm} 1 = c_K X + O\left(X^{1-1/d}\right),\end{equation} where $c_K$ is defined as above in (1.\ref{c_K}) and $d = [K: \Q].$


\subsection{Smooth ideals and the Dickman function}\label{integral}
The Dickman function $\rho(\beta)$ is defined to be the continuous solution to the equations:
\begin{align*}
\rho(\beta) &= 1, \textrm{ for } 0 \leq \beta \leq 1\\
-\beta\rho'(\beta) &= \rho(\beta - 1), \textrm{ for } \beta > 1.
\end{align*}

This function appears classically in counting smooth numbers, numbers whose prime divisors are bounded. We use it to count the number-field analogue, ``smooth ideals.'' We define $\Ss(X, Y) \coloneqq \set{I \subset \O_K \, \colon N(I) \leq X,\; M(I) \leq Y},$ and we denote $\Psi(X, Y) \coloneqq \#\Ss(X, Y).$ Moree \cite{Moree} proved that
\begin{equation}\numberwithin{equation}{section}\label{Moreelem}
\Psi(X, Y) = X\rho(\beta) \left(1 + O_\varepsilon \left(\frac{\beta\log(\beta + 1)}{\log X}\right)\right),
\end{equation} where $\beta = \log X / \log Y$ and the asymptotic is uniform for $1 \leq \beta \leq (\log X)^{1 - \varepsilon}.$

The Dickman function is easily bounded above  by $1/\Gamma(\beta + 1),$ which yields a uniform asymptotic \begin{equation}\label{rhobound}\rho(\beta) = O\left(\beta^{-1/2} \exp\set{-\beta \log \beta / e}\right).\end{equation}

\subsection{Classical bounds in the number field setting}\label{bounds}
In this section, we prove some number field analogues of well-known bounds that we will need in our proofs. 
We also estimate several partial sums involving the M\"{o}bius function. Our first bounds are given by the following lemma. \begin{lemma}\label{mobiussum}
As $X\to\infty$, we have: \begin{align*}&(1) \; \;\; \sum_{N(I) \leq X} \mu(I) = O
\left(X\exp\set{-k(\log X)^{1/12}}\right). \\
&(2) \; \;\; \sum_{N(I)\leq X} \frac{\mu(I)}{N(I)} = O\left(\exp\set{-k(\log X)^{1/12}}\right).\end{align*}
\end{lemma}

\begin{proof}
(1) The case $K = \Q$ was proven by Landau \cite{Landaubook}. The adaptation to general $K$ is nearly immediate; the proof for the $\Q$ case relies on the zero-free region for $\zeta(s)$, and an analogous zero-free region exists for the Dedekind zeta function $\zeta_K(s)$ \cite{Landauart}. 

(2) This follows from (1) by partial summation. 
\end{proof}

Our next lemma requires a more complicated estimate.

\begin{lemma}\label{mulog/n}
As $X \to \infty$, we have that \[
\sum_{N(I)\leq X} \frac{\mu(I)\log N(I)}{N(I)} = -\frac{1}{c_K} + O\left(\exp\set{-k(\log X)^{1/12}}\right).
\]
\end{lemma}
\begin{proof}
Our starting point is the Prime Ideal Theorem. Partial summation yields: 
\begin{align}\label{blah1}
\theta(X) \coloneqq \sum_{N(\p)\leq X} \log N(\p) &= X + O\left(X\exp\set{-k(\log X)^{1/2}}\right).
\end{align}
We define the Von Mangoldt function for ideals by: \[
\Lambda(I) = \begin{cases}\log N(\p), & \textrm{if }I = \p^k \textrm{ for a prime ideal }\p \textrm{ and }k\geq 1\\ 0,& \textrm{otherwise.}
\end{cases}
\] 
Since the number of prime ideals lying above a given prime of $\Q$ is bounded by the degree $[K:\Q]$, we can extend the classical result
\begin{equation}\label{blah2}
\psi(X) \coloneqq \sum_{N(I)\leq X} \Lambda(I) = \theta(X) + O(\sqrt{X}).
\end{equation} 
Combining (\ref{blah1}) and (\ref{blah2}), we conclude that \begin{equation}\label{psiasym}\psi(X) = X + O\left(X\exp\set{-k(\log X)^{1/2}}\right).\end{equation}
Now we will use the identity $\Lambda = -\mu\log * 1$, where $*$ denotes Dirichlet convolution: 
\begin{align*}
\psi(x) &= \sum_{N(I)\leq X} \sum_{J\supset I} -\mu(J) \log N(J) \\
&= \sum_{N(J)\leq X} -\mu(J)\log N(J) \left( c_K \cdot \frac{X}{N(J)} + O\left(\left(\frac{X}{N(J)}\right)^{1-1/d}\right)\right) \\
&= c_K X \sum_{N(J)\leq X} \frac{-\mu(J)\log(N(J))}{N(J)} + O\left(X^{1-1/d} \sum_{N(I)\leq X} \frac{-\mu(J)\log N(J)}{N(J)^{1-1/d}}\right). \\
\end{align*} Then partial summation combined with Lemma \ref{mobiussum} gives us that 
$$\sum_{N(I)\leq X} \frac{\mu(I)\log(I)}{N(J)^{1-1/d}} = O\left(X^{1/d} \exp\set{-k(\log X)^{1/12}}\right).$$


\noindent Putting our estimates together, we obtain: 
\begin{align*}
\psi(X) = c_K X \sum_{N(J)\leq X} \frac{-\mu(J)\log(J)}{N(J)} + O\left(X\exp\set{-k(\log X)^{1/12}}\right).
\end{align*}
Setting this equal to our earlier estimate of $\psi$ in (\ref{psiasym}) gives the lemma.
\end{proof}

\section{Proof of Theorem 1.1}\label{proof1}
\noindent Throughout the remainder of this paper, define $Y$ to vary with $X$ by $Y \coloneqq \exp\set{(\log X)^{2/3}}$. 
We also define a maximal prime counting function by \[Q(I) \coloneqq \#\set{I \subset \p \,\colon N(\p) = M(I)}.\] The proof of Theorem \ref{thm1} is structured as a series of lemmas. Our ultimate goal is to estimate $\sum Q_C(I)$ by comparing it to $\sum Q(I)$; to do so, we will split the sum into sums over small and large primes. (This technique is an adaptation of Dawsey \cite{Dawsey}.) We begin with a bound that we will use frequently to show that certain sums over small primes are negligible.

\begin{lemma}\label{smallprimes}
As $X\rightarrow \infty$, we have that
\[\sum_{N(\p) \leq Y} \Psi(X/N(\p), N(\p)) = O\left(X\exp\set{-k(\log X)^{1/3}}\right).\]
\end{lemma}

\begin{proof}
For every prime ideal $\p$, where $N(\p)\leq Y$, and for every $I \in \Ss(X/N(\p), N(\p))$, we obtain an ideal $I\p \in \Ss(X, Y)$. The number of pairs $(\p, I)$ that give the same $I\p$ is at most $[K:\Q]$, since only $[K:\Q]$ ideals can possibly have norm $N(\p)$. Therefore \[\sum_{N(\p) \leq Y} \Psi(X/N(\p), N(\p)) \leq [K:\Q] \Psi(X, Y).\] Now we apply the estimate for $\Psi(X, Y)$ in (\ref{Moreelem}).
\end{proof}

We use this lemma to show that the sums $\sum Q_C(I)$ and $\sum Q(I)$ are well-approximated by the same integral; this will be useful because $\sum Q(I)$ is easier to understand. 

\begin{lemma}\label{QCisint}
As $X\rightarrow \infty$, we have that:
\begin{align*}&(1) \;\;\; \sum_{2 \leq N(I) \leq X} Q_C(I) =  \frac{|C|}{|G|}\int_Y^X \Psi(X/t, t) \frac{dt}{\log t} + O\left(X \exp\set{-k (\log X)^{1/3}}\right). \\&(2) \; \; \; \sum_{2\leq N(I) \leq X} Q(I) =  \int_Y^X \Psi(X/t, t) \frac{dt}{\log t} + O\left(X \exp\set{-k (\log X)^{1/3}}\right).\end{align*}
\end{lemma}
\begin{proof}
We only prove (1), since the proofs are essentially identical. For the case (2), it suffices to replace all $Q_C(I)$ by $Q(I)$ and remove all conditions $\artin{\p} = C$ from sums.

A simple counting argument shows 
\[\sum_{2 \leq N(I) \leq X} Q_C(I) = \sum_{\substack{N(\p) \leq X \\ \left[\frac{L/K}{\p}\right] = C}} \Psi(X/N(\p), N(\p)).\]
We split our sum into small and large primes. Then \[\sum_{2 \leq N(I) \leq X} Q_C(I)  = \sum_{\substack{N(\p) \leq Y \\ \left[\frac{L/K}{\p}\right] = C}} \Psi(X/N(\p), N(\p)) + \sum_{\substack{Y < N(\p) \leq X \\ \left[\frac{L/K}{\p}\right] = C}} \Psi(X/N(\p), N(\p)).\] The first term is clearly $O\left(X \exp\set{-k (\log X)^{1/3}}\right)$ by Lemma \ref{smallprimes}.

We wish to estimate the second sum by an integral. We study the difference, defined by \[E \coloneqq \sum_{\substack{Y < N(\p) \leq X\\ \artin{\p} = C}} \Psi(X/N(\p), N(\p)) - \frac{|C|}{|G|}\int_Y^X \Psi(X/t, t) \frac{dt}{\log t}.\] We rewrite $E$ using the definition of $\Psi$ and then switch orders of summation and integration:
\begin{align*}
E &= \sum_{\substack{Y \leq N(\p) \leq X \\ \artin{\p} = C}} \sum_{I \in \Ss(X/N(\p), N(\p))} 1 - \frac{|C|}{|G|} \int_Y^X \sum_{I \in \Ss(X/N(\p), N(\p))} 1 \frac{dt}{\log t} \\
&= \sum_{\substack{1 \leq N(I) \leq X/ Y\\ M(I) \leq X/N(I)}}\Bigg(\sum_{\substack{M(I) \leq N(\p) \leq X/N(I) \\ N(\p) > Y \\ \artin{\p} = C}} 1 - \frac{|C|}{|G|}\int_{\max(M(I), Y)}^{X/N(I)} \frac{dt}{\log t}\Bigg)\\
&= \sum_{\substack{1 \leq N(I) \leq X/ Y\\ M(I) \leq X/N(I)}} \Bigg(\pi_C(L/K; X/N(I)) - \pi_C(L/K; \max(M(I), Y)) \\ & \hspace{4cm}- \frac{|C|}{|G|} \Li (X/N(I)) + \frac{|C|}{|G|} \Li(\max(M(I),Y))\Bigg).\\
\end{align*} 
The Chebotarev Density Theorem bounds \[|E| \leq  \sum_{\substack{1 \leq N(I) \leq X/ Y\\ M(I) \leq X/N(I)}} c_1(X/N(I))\exp\set{-c_2\sqrt{\log(X/N(I))}},\] where we use the fact that $\max(M(I), Y) \leq X/N(I)$ for all $I$ in the sum. (To show (2), we apply the Prime Ideal Theorem instead of the Chebotarev Density Theorem.) At this point, we bound \[\exp\set{-c_2\sqrt{\log (X/n)}} \leq \exp\set{-k (\log X)^{1/3}},\] which yields: \begin{align*}E & = O\left(X\exp\set{-k (\log X)^{1/3}}\right).\end{align*} 
Note that we have used a bound on the sum $1/N(I)$ which can be easily deduced by partial summation with (\ref{idealcount}); alternatively, it can be found in \cite{Shapiro}.

We conclude that \[\sum_{2\leq N(I) \leq X} Q_C(I) = \frac{|C|}{|G|} \int_Y^X \Psi(X/t, t) \frac{dt}{\log t} + O\left(X\exp\set{-k(\log X)^{1/3}}\right).\]

\end{proof}
The next step in the proof of Theorem 1.1 is to estimate $\sum Q(I)$. 

\begin{lemma}\label{QisX} As $X\to \infty$, the following asymptotic holds:
\[\sum_{2\leq N(I) \leq X} Q(I) = c_K X + O\left(X \exp\set{-k(\log X)^{1/3}}\right).\]
\end{lemma}
\begin{proof}
 Write \[\sum_{2 \leq N(I) \leq X} Q(I) = \sum_{2 \leq N(I) \leq X} 1 + \sum_{\substack{2 \leq N(I) \leq X\\ Q(I) \geq 2}} (Q(I) - 1).\] Now, the first term, by (\ref{idealcount}), is \[c_K X + O(X^{1 - 1/d}),\] where $d = [K: \Q]$. This error term is better than required for the lemma, so we need only consider the second term. Observe that:

\[\sum_{\substack{2 \leq N(I) \leq X \\Q(I) \geq 2}} (Q(I) - 1) \leq  d^2 \sum_{N(\p) \leq X} \Psi(X/N(\p)^2, N(\p)).\] As before, we split our sum into small and large primes: \[\sum_{N(\p) \leq X} \Psi(X/N(\p)^2, N(\p)) = \sum_{N(\p) \leq Y} \Psi(X/N(\p)^2, N(\p)) + \sum_{Y < N(\p) \leq \sqrt{X}} \Psi(X/N(\p)^2, N(\p)).\] Since we certainly have the inequality $\Psi(X/N(\p)^2, N(\p)) \leq \Psi(X/N(\p), N(\p)),$ we can apply Lemma \ref{smallprimes} and restrict our attention to the sum over large primes. Next, we use a similar integral trick to that in the proof of Lemma \ref{QCisint}. Write \[E' \coloneqq \sum_{Y < N(\p) \leq \sqrt{X}} \Psi(X/N(\p)^2, N(\p)) - \int_Y^{\sqrt{X}} \Psi(X/N(\p)^2, N(\p)).\] We switch orders of summation and integration to write: \begin{align*}
E' &= \sum_{\substack{1 \leq N(I) \leq X/Y^2\\M(I) \leq \sqrt{X/N(I)}}} \Bigg(\sum_{\substack{M(I) \leq N(\p) \leq \sqrt{X/N(I)} \\ N(\p) > Y}} 1 - \int_{\max(M(I), Y)}^{\sqrt{X/N(I)}} \frac{dt}{\log t} \Bigg) \\
&= \sum_{\substack{1 \leq N(I) \leq X/Y^2\\M(I) \leq \sqrt{X/N(I)}}} \bigg(\pi_K\Big(\sqrt{X/N(I)}\Big) - \pi_K(\max(M(I), Y)) \\ &\hspace{4cm}- \Li \Big(\sqrt{X/N(I)}\Big) + \Li(\max(M(I), Y))\bigg).
\end{align*} At this point, the Prime Ideal Theorem tells us that 
\begin{align*}
E' &= O\Bigg( \sum_{\substack{1 \leq N(I) \leq X/Y^2\\M(I) \leq \sqrt{X/N(I)}}}  \sqrt{X/N(I)} \exp\set{-k(\log(X/N(I)))^{1/2}}\Bigg).
\end{align*} To bound this sum, we replace $(\log(X/N(I)))^{1/2}$ by $(\log X)^{1/3}$ at the cost of a constant, since $X/N(I) \geq Y^2 = \exp\set{2 (\log (X))^{2/3}}.$ Using partial summation, we find that \[E' = O\left(X \exp\set{-k(\log X)^{1/3}}\right).\] As a result, we obtain: \[\sum_{2 \leq N(I) \leq X} Q(I)  = c_K X + \int_Y^{\sqrt{X}} \Psi( X/t^2, t) \frac{dt}{\log t} + O\left(X\exp\set{-(\log X)^{1/3}}\right).\]

Finally, we estimate the integral \[\int_Y^{\sqrt{X}} \Psi(X/t^2, t) \frac{dt}{\log t}.\] We will bound this integral using (\ref{Moreelem}), the asymptotic bound for $\Psi(X, Y)$. Set $\beta = \log X / \log t$. Note the asymptotic bound applies uniformly for $Y \leq t \leq \sqrt{X}$, taking $\varepsilon = 2/3$. Indeed, we find that

\begin{align*}
\int_Y^{\sqrt{X}} \Psi(X/t^2) \frac{dt}{\log t} &\leq k\int_Y^{\sqrt{X}} \frac{X}{t^2} \exp\set{-(\beta - 2)\log(\beta - 2)/e} \frac{dt}{\log t}\\
&= k\int_2^{(\log X)^{1/3}} \frac{X\log t}{t \log X} \exp\set{-(\beta - 2)\log(\beta - 2)/e} d\beta \\
&= k\int_2^{(\log X)^{1/3}} \frac{X}{\beta} \exp\set{-\log X / \beta} \exp\set{-(\beta - 2)\log(\beta - 2)/2} d\beta \\
&\leq k X \exp\set{-(\log X)^{2/3}} \int_2^{\infty} \exp\set{-(\beta - 2)\log(\beta - 2)/2} \frac{d\beta}{\beta} \\
&= O\left(X \exp\set{-(\log X)^{2/3}}\right).
\end{align*}

\end{proof}
\begin{proof}[Proof of Theorem 1.1]

Lemma \ref{QCisint} (2) and Lemma \ref{QisX} imply that 
\[\int_Y^X \Psi(X/t, t) \frac{dt}{\log t} = c_K X +  O\left(X \exp\set{-k (\log X)^{1/3}}\right).\]

We therefore conclude that \[\sum_{2 \leq N(I) \leq X} Q_C(I) = c_K\cdot \frac{| C|}{| G|}  X + O\left( X \exp\set{-k (\log X)^{1/3}}\right).\]

\end{proof}
\section{Proof of Theorem 1.2}\label{proof2}
\noindent We begin with a lemma that reformulates Theorem 1.1 in a more useful form to apply the duality result of Lemma \ref{duality}. 

\begin{lemma}\label{lemma4} As $X\to\infty$, we have:
\[
\sum_{2\leq N(I)\leq X} \frac{Q_C(I)}{N(I)} = c_K\cdot \frac{|C|}{|G|}\log(X) + k' + O\left(\exp\set{-k(\log X)^{1/3}}\right),
\] where $k'$ is a constant depending only on $L$, $K$, and $C$.
\end{lemma}

\begin{proof}
We begin from Theorem 1.1, which asserts that \[\sum_{2 \leq N(I) \leq X} Q_C(I) = c_K\cdot \frac{|C|}{|G|}X + O\left(X\exp\set{-k (\log X)^{1/3}}\right).\] For ease of presentation, we define \[T(X) \coloneqq \sum_{2 \leq N(I) \leq X} Q_C(I)\] and \[E(X) \coloneqq T(X) - c_K \cdot \frac{|C|}{|G|} X.\] Thinking of $T$ as a stair-step function, we compute:
\begin{align*}
\sum_{2\leq N(I)\leq X} \frac{Q_C(I)}{N(I)} &= \int_2^X \frac{dT(t)}{t}\\
&= c_K\cdot \frac{|C|}{|G|} \int_2^X \frac{dt}{t} + \int_2^X \frac{dE(t)}{t} \\
&= c_K \cdot \frac{|C|}{|G|} \log(X) + \frac{E(t)}{t}\bigg|_2^X - \int_2^X \frac{E(t)}{t^2} dt,
\end{align*} where the last two terms are the error we wish to bound.
Since $E(X) = O\left(X\exp\set{-k (\log X)^{1/3}}\right)$, we have that \[\int_2^\infty \frac{E(t)}{t^2} dt < \infty.\] Therefore we rewrite:
\begin{align*}
\sum_{2\leq N(I) \leq X} \frac{Q_C(I)}{N(I)} &= c_K\cdot  \frac{|C|}{|G|} \log (X) + \frac{E(X)}{X} + \int_X^\infty \frac{E(t)}{t^2} dt + k'\\
&= c_K \cdot  \frac{|C|}{|G|} \log(X) + k' + O\left(\exp\set{-k (\log X)^{1/3}}\right).
\end{align*}
\end{proof}

\begin{proof}[Proof of Theorem 1.2]
Now we consider the sum in the theorem statement. Using the definition of $f$ in Lemma \ref{duality}, we rewrite 
\[
\sum_{\substack{2\leq N(I)\leq X \\ I\in S(L/K;C)}}  \frac{\mu(I)}{N(I)} = \sum_{2\leq N(I)\leq X} \frac{\mu(I)f(I)}{N(I)}. \]
Using Lemma \ref{duality}, this is equivalent to:
\[
-\sum_{N(I)\leq X}\frac{1}{N(I)}\sum_{J\supset I} \mu(I/J)Q_C(J) = -\sum_{N(I)\leq X}\sum_{J\supset I} \frac{\mu(I/J)}{N(I)/N(J)} \cdot \frac{Q_C(J)}{N(J)}.
\] We split this sum into two parts. Let $I/J = I'$. Given an ideal $I'$, the possible $J$ so that $I'J = I$ are those with $N(J) \leq X/N(I')$. We split the sum into sums over $N(I') \leq \sqrt{X}$ and $\sqrt{X} < N(I') \leq X$:
\begin{align*}
    -\sum_{\substack{1 \leq N(I') \leq \sqrt{X}\\ N(J) \leq X / N(I') }} \frac{\mu(I')}{N(I')} \cdot \frac{Q_C(J)}{N(J)} - \sum_{\substack{\sqrt{X} < N(I') \leq X \\ N(J) \leq X / N(I')}} \frac{\mu(I')}{N(I')} \cdot \frac{Q_C(J)}{N(J)}.
\end{align*}

We call these terms $S_1$ and $S_2$, respectively. For readability, we  write $I'$ as $I$. We first analyze $S_1$, which will turn out to give the main term. We apply Lemma \ref{lemma4}. 

\begin{align*}
-S_1 &= \sum_{\substack{1 \leq N(I) \leq \sqrt{X} \\ N(J) \leq X / N(I)}} \frac{\mu(I)}{N(I)} \cdot \frac{Q_C(J)}{N(J)} \\
&= \sum_{N(I) \leq \sqrt{X}} \frac{\mu(I)}{N(I)}  \left(c_K\cdot \frac{|C|}{| G|}  \log \Big(\frac{X}{N(I)}\Big) +k' + O\left(\exp\set{-k(\log X)^{1/3}}\right)\right). \\
\end{align*}
Now we use Lemma \ref{mobiussum} to write:
\begin{align*}
-S_1 =  \left(c_K\cdot \frac{|C|}{| G|} \log X\right)& \sum_{N(I) \leq \sqrt{X}} \frac{\mu(I)}{N(I)} \\ &- c_K \cdot  \frac{|C|}{|G|} \sum_{N(I) \leq \sqrt{X}} \frac{\mu(I) \log N(I)}{N(I)} + O\left(\exp\set{-k(\log X)^{1/12}}\right);
\end{align*}
whence, applying Lemmas \ref{mobiussum} and \ref{mulog/n}, we conclude that \[S_1 = -\frac{|C|}{|G|} + O\left(\exp\set{-k (\log X)^{1/12}}\right).\] Now we bound $S_2$. We can switch the order of summation to write $-S_2$ as 
\begin{align*}
-S_2 = \sum_{N(J) < \sqrt{X}} \frac{Q_C(J)}{N(J)} \sum_{\sqrt{X} < N(I) \leq X/N(J)} \frac{\mu(I)}{N(I)}.
\end{align*}
Applying Lemma \ref{mobiussum} and using that $Q_C(J)$ is bounded by the degree of the extension, we have that 
\begin{align*}
S_2 &= O\left(\exp\set{-k (\log (X)^{1/12}}\right).
\end{align*}

We conclude that \[\sum_{\substack{N(I) \leq X\\ I\in S(L/K;C)}} \frac{\mu(I) f(I)}{N(I)} = -\frac{| C|}{| G|} + O\left(\exp \set{-k (\log X)^{1/12}}\right).\]
Taking the limit, we finally obtain the desired conclusion: \[-\lim_{X \to \infty} \sum_{\substack{N(I) \leq X\\ I\in S(L/K;C)}} \frac{\mu(I)}{N(I)} = \frac{| C|}{| G|}.\]

\end{proof}

\section{Examples}
\noindent Here we illustrate Theorem 1.2 with two examples. As these examples will show, the convergence of the limit can be quite slow. 

\noindent \textbf{Example 1.} First we consider a quadratic extension. Let $K = \Q(\zeta_7)$, where $\zeta_7$ is a primitive $7$th root of unity, and $L = \Q(\zeta_7, \sqrt{2})$. Then we know that $\Gal(L/K) \simeq \Z/2\Z$. Since this group is abelian, the conjugacy classes are in one-to-one correspondance with the elements. 
We denote \[S_C(X) \coloneqq -\sum_{\substack{2\leq N(I)\leq X\\ I\in S(L/K;C)}} \frac{\mu(I)}{N(I)}.\]
According to Theorem 1.2, we have that
\[
\lim_{X\rightarrow\infty} S_{\set{0}}(X) = \lim_{X\rightarrow\infty} S_{\set{1}}(X) = \frac{1}{2}. 
\]Indeed, we compute the following actual values of the sums for increasing values of $X$. As we can see, the convergence is quite slow and subject to fluctuation.

\begin{center}
\begin{tabular}{|c|c|c|}
\hline
    & $\{0\}$ & $\{1\}$ \\
 \hline
$S_C(10,000)$  & 0.4709 &0.5117 \\
\hline
$S_C(20,000)$ & 0.4931& 0.5075\\
\hline
$S_C(30,000)$ & 0.5032 & 0.5041 \\
\hline
$S_C(40,000)$ & 0.5042 & 0.5062\\
\Xhline{2\arrayrulewidth}
$|C|/|G|$ & 0.5000 & 0.5000 \\
\hline 
\end{tabular}
    
\end{center}
\medskip 

\noindent \textbf{Example 2.} Next, we consider a nonabelian case. Let $K=\Q(i)$ and let $L$ be the splitting field of $x^3+x+1$ over $K$. It is easy to check that $\Gal(L/K) \simeq S_3$, so the conjugacy classes are the different cycle compositions: the identity, transpositions, and 3-cycles.  A computer can easily determine the conjugacy class $\artin{\p}$ by factoring $x^3 + x + 1$ over the finite field $\O_K/\p$.   The table below compares $S_C(X)$ to the limits predicted by Theorem 1.2 in increments up to $500,000$. 
\begin{center}
    \begin{tabular}{|c|c|c|c|c|c|}
        \hline
        & [(1)] &  [(12)] & [(123)] \\

        \hline
        $S_C(100,000)$ & 0.1420 & 0.5268 & 0.3376 \\
        \hline
        $S_C(200,000)$ & 0.1461 & 0.5213 & 0.3380 \\
        \hline
        $S_C(300,000)$ & 0.1485 & 0.5183 & 0.3374 \\
        \hline
        $S_C(400,000)$ & 0.1499 & 0.5164 & 0.3374 \\
        \hline 
        $S_C(500,000) $ & 0.1510 & 0.5151& 0.3374  \\
\Xhline{2\arrayrulewidth}
        $|C|/|G|$ & 0.1667 & 0.5000 & 0.3333  \\
        \hline
    \end{tabular}
\end{center}

\section*{Acknowledgements}
\noindent The authors would like to thank Professor Ken Ono and Professor Larry Rolen for their guidance and suggestions. They also thank Emory University, the Asa Griggs Candler Fund, and NSF grant DMS-1557960.

\bibliographystyle{abbrv}
\bibliography{biblio}

\end{document}